\documentclass[a4paper]{amsart}
\usepackage[utf8]{inputenc}
\usepackage[english]{babel}

\usepackage{amsmath, amsthm, amssymb, amsfonts}
\usepackage[dvipsnames]{xcolor}
\usepackage{dsfont}
\usepackage[bbgreekl]{mathbbol}
\usepackage{csquotes}
\usepackage{enumitem}
\usepackage{extarrows}
\usepackage{leftidx}
\usepackage[arrow, matrix, curve]{xy}
\usepackage{pdfpages}
\usepackage[most]{tcolorbox}
\usepackage{ifthen}
\usepackage{cases}
\usepackage{enumitem}

  \newcommand{\Z}{\mathds{Z}}
  \newcommand{\C}{\mathds{C}}
  \newcommand{\Q}{\mathds{Q}}
  \newcommand{\R}{\mathds{R}}
  \newcommand{\F}[1]{\mathds{F}_{#1}}
  \newcommand{\Fq}{{\F{q}}}

  \newcommand{\GL}{\mathop{\mathrm{GL}}}
  \newcommand{\Aut}[2]{\mathrm{Aut}_{#1}(#2)}
  \newcommand{\rk}{{\mathop{\mathrm{rk}}}}
  \newcommand{\Spec}{\mathop{\mathrm{Spec}}}
  \newcommand{\etale}{{\'{e}tale}}
  \newcommand{\Ker}{{\mathop{\mathrm{ker}}}}

  \newcommand{\QEnd}[1]{{\mathrm{QEnd}(#1)}}
  \newcommand{\End}[2]{{\EndCat{#1}{#2}}}
  \newcommand{\EndCat}[2]{{\mathrm{End}_{#1}(#2)}}
  \newcommand{\Hom}[3]{\mathrm{Hom}_{#3}(#1,#2)}
  \newcommand{\QHom}[3]{\mathrm{QHom}_{#3}(#1,#2)}

  \newcommand{\Am}{$ A $-motive}

  \newcommand{\M}{{\underline{\hat{M}}}} 
  \newcommand{\N}{\underline{\hat{N}}} 
  \newcommand{\Carlitz}{{\underline{\hat{C}}}}
  \newcommand{\MC}{{\mathcal{C}}}
  \newcommand{\MO}[1]{\mathcal{O}_{#1}}
  \newcommand{\MOC}{\MO{\MC}}

  \newcommand{\Laurrentreihe}[1]{\dbl #1\dbr}
  \newcommand{\DoppelRund}[1]{{(\!(#1)\!)}}
  \newcommand{\DoppelSpitz}[1]{{\langle\!\langle#1\rangle\!\rangle}}
  \newcommand{\Heins}[2]{\text{H}^1_{#1}(#2)}

  \newcommand{\GM}[1]{\mathbb{G}_{m #1}}
  
  \newcommand{\MG}[1]{\Gamma_{#1}} 
  \newcommand{\iMG}[1]{\Gamma_{#1}^i}
  
  \newcommand{\Gal}{\mathrm{Gal}}
  \newcommand{\GR}[1]{\rho_{#1}}
  \newcommand{\AbsGalois}[1]{\mathcal{G}_{#1}}

  \newcommand{\KatLshtuka}[1]{\texttt{Sht}_{#1}}
  \newcommand{\KatRep}[3]{\mbox{\texttt{Rep}}_{#1}^{\text{\tiny#3}}(#2)}
  \newcommand{\Erzeugt}[1]{\langle#1\rangle}
  
  \newcommand{\TensorAut}[2]{\mathrm{Aut}^{\otimes}(#1|#2)}
  
  \newcommand{\KatVec}[1]{\text{\texttt{Vec}}_{#1}}
  \newcommand{\KatMod}[1]{\text{\texttt{Mod}}_{#1}}

  \newcommand{\dsim}{\rotatebox[origin=c]{-90}{$\sim$}}

  \newcommand{\dbl}{{\mathchoice{\mbox{\rm [\hspace{-0.15em}[}}
  		{\mbox{\rm [\hspace{-0.15em}[}}
  		{\mbox{\scriptsize\rm [\hspace{-0.15em}[}}
  		{\mbox{\tiny\rm [\hspace{-0.15em}[}}}}
  \newcommand{\dbr}{{\mathchoice{\mbox{\rm ]\hspace{-0.15em}]}}
  		{\mbox{\rm ]\hspace{-0.15em}]}}
  		{\mbox{\scriptsize\rm ]\hspace{-0.15em}]}}
  		{\mbox{\tiny\rm ]\hspace{-0.15em}]}}}}

  \newcommand{\TateVFunc}[1]{ \check{V}_v{#1}}
  \newcommand{\TateTFunc}[1]{ \check{T}_v{#1}}
  \newcommand{\HalgGp}[1]{ \text{H}_{#1}}
  \newcommand{\Fqt}{\F{q}[t]}
  \newcommand{\Fqrt}{\F{q}(t)}
  \newcommand{\Fqz}{\F{q}\Laurrentreihe{z}}

  \newcommand{\Fqrz}{\F{q}\DoppelRund{z}}
  
  \newcommand{\Fqrzeta}{\F{q}\DoppelRund{\zeta}}

  \newcommand{\Rz}{R\Laurrentreihe{z}}

  \newcommand{\DrinfeldStrata}[1][]{%
      \ifthenelse{ \equal{#1}{} }
        {\ensuremath{M_{\mathcal{n},\F{\nu}}^{r}}}
        {\ensuremath{M_{\mathcal{n},\F{\nu}}^{r,{#1}}}}
      }

  \newcommand\restrict[1]{\raisebox{-.5ex}{$|$}_{#1}}

  \usepackage{pifont}

  \definecolor{db}{RGB}{23,20,119}
  \newtheoremstyle{Bunt}{}{}{\color{db}}{}{\color{db}\bfseries}{}{ }{}
  \newtheoremstyle{test}
  {10 pt}
  {10 pt}
  {}
  {0pt}
  {\bfseries}
  {. }
  { }
  {\thmname{#1}\thmnumber{ #2}\thmnote{ (#3)}}

  \newtheoremstyle{heroem}
  {\parindent}
  {\parindent}
  {\itshape}
  {0pt}
  {\bfseries}
  {}
  {1pt plus 1pt minus 1pt }
  {}

  \theoremstyle{definition}
  \newtheorem{defi}{Definition}[section]
  \theoremstyle{plain}
  \newtheorem{prop}[defi]{Proposition}
  \newtheorem{Theorem}[defi]{Theorem}
  \theoremstyle{heroem}
  \newtheorem*{heorem}{Theorem}
  \newtheorem*{rop}{Proposition}
  \newtheorem*{orollary}{Corollary}

  \theoremstyle{plain}
  \newtheorem{Corollary}[defi]{Corollary}
  \newtheorem{Lemma}[defi]{Lemma}
  \theoremstyle{remark}
  \newtheorem{rem}[defi]{Remark}
  \theoremstyle{test}
  \newtheorem{ex}[defi]{Example}

  \title{Non-openness of v-adic Galois representation for A-motives}
  \author{Maike Ella Elisabeth Frantzen}

\usepackage{hyperref}
\usepackage{cleveref}

\begin{document}
	\setlist[enumerate]{itemsep=0pt}
  \begin{abstract}
  	Drinfeld-modules and \Am s are the function field analogous of elliptic curves and abelian varieties. For the latter one  can construct the $l$-adic Galois representation and can ask if  its image is open. For Drinfeld-modules this question was answered by Pink and his coauthors. We clarify the rank one case and show that the image of Galois is open if and only if the virtual dimension is prime to the characteristic of the ground field.
  \end{abstract}
	\maketitle
	\setcounter{tocdepth}{1}
  \setlength{\parindent}{0em}
	\tableofcontents
  \section{Introduction}
  Let $A$ be an abelian variety of dimension $d $ over a number field $K$, where $K$ is fixed inside $\C$. For  any rational prime $l$, the  $l$-adic Tate-module of $A$ is given by $T_lA= \varprojlim A[l^{n+1}](K^{sep})$, for $K^{sep}$ being a separable closure of $K$. It is a free $\Z_l$-module of rank $2d$ and carries a natural $\AbsGalois{K}:=\mathrm{Gal}(K^{sep}/K)$-action. The rational $l$-adic Tate-module is the $2d$-dimensional vector space over $\Q_l$ given by $V_lA:= T_lA \otimes_{\Z_l} \Q_l$.
  This defines a natural representation
  \begin{align*}
  \GR{l}:\AbsGalois{K} \to \mathrm{Aut}(T_lA)
  \end{align*}
  contained in $\mathrm{Aut}(V_lA)\cong\mathrm{GL}_{2d}(\Q_l)$. This is called the $l$-adic Galois representation attached to an abelian variety $A$.
  Another algebraic group associated to $A$ is the \emph{Mumford-Tate} group $MT(A)$. It is  a $\Q$-subgroup of $\mathrm{GL}(V)$, where $ V := H_1(A(\C),\Q)$ is the first rational homology group of $A$.
  The \emph{Mumford-Tate conjecture} asks if the Zariski closure of $\GR{l}(\AbsGalois{K})$ is equal to the Mumford-Tate group $ MT(A) $ over $\Q_l$. One can furthermore ask,  if an open subgroup of  $\GR{l}(\AbsGalois{K})$  is also an open subgroup of the $\Q_l$-valued Mumford Tate group.\\
  The same question can be asked in the theory function field arithmetic. Here elliptic curves and abelian varieties are replaced by Drinfeld modules and $A$-motives. In the Drinfeld case, the above question was answered by Richard Pink in~\cite{Pink.1997}. Here he proves the Mumford-Tate conjecture for Drinfeld modules by showing the openness for the Galois representation.  As higher generalization of Drinfeld modules, Greg Anderson defined in~\cite{Anderson.1986} abelian $t$-modules and their dual notion of $t$-motives. For these objects Andreas Maurischat   gives a result on Galois openness in~\cite{Maurischat.2019}.
  There, he  defines a density for the representation~\cite[definition 3.1]{Maurischat.2019} and relies the $t$-adically non-openness to this density~\cite[proposition 3.2]{Maurischat.2019}. Note that our main example, the Carlitz-module~\ref{ex: Carlitz introduction}, has density $1$, which coincides with our result.\\
  These objects can be slightly generalized to \Am s.
  To define them,  let $\Fq$ be a finite filed with $q$ elements and odd prime characteristic $p$. Then let $\MC$ be a smooth projective and geometrically irreducible curve over $ \F{q} $ and  let $ \infty $ be a fixed closed point of $\MC$. Denote by  $ A:=\Gamma(\MC\setminus \{\infty\}, \MOC) $ the ring of regular functions of $ \MC $ outside $ \infty $ and by $ Q=\F{q}(\mathcal{C}) $ its fraction field.\\
  Let $ L $ be a   finite field extension of $ Q $, fixed in the algebraic closure $ Q^{alg} $ of $ Q $ and set $ A_L= A \otimes_{\Fq} L $ and $ \sigma = id_A \otimes Frob_{q,L} $.  We consider for an $ A_L $-module $ M $  the pullback   $ \sigma^*M= M\otimes_{A_L, \sigma} A_L$ and for a $A_L $-homomorphism $ f:M\to N $  the pullback homomorphism
  $ \sigma^*f:= f \otimes id_{A_L}:\sigma^*M\to \sigma^*N$.  Set $ \gamma:A \to L  $ for the embedding of $ A\subset Q\subset L $ and consider $ J= (a\otimes1-1\otimes\gamma(a)|a\in A )$ as the maximal ideal in $ A_L $. Then $ \gamma $  can be reconstructed as $ A\to A_L/ J = L $ with $  a\mapsto a \otimes 1 (\mathrm{mod} J )$. Thus we define the  \emph{$ A $-characteristic of $ L $} to be the ideal $ \epsilon:=\ker\gamma \subseteq A$.
  \begin{defi}\label{def:Amot}
  	An \emph{\Am \ $ \underline{M} $ over $ L $} is a pair $ (M, \tau_M) $ consisting  of a locally free $ A_L $-module $ M $ of rank $ r $ and an isomorphism
  	\begin{align*}
  	\tau_M: \sigma^*M\lvert_{\Spec A_L \backslash V(J)} \to M\lvert_{\Spec A_L \backslash V(J)}
  	\end{align*} of the associated sheaves outside $ V(J)\subset \Spec A_L $.\\
  	The \emph{rank} of $ \underline{M} $ is the rank of the $ A_L $-module $ M $ and  denoted by  $ \rk \underline{M} $. The \emph{(virtual) dimension} $\dim\underline{M}$ of $\underline{M}$ is defined as
  	\begin{align*}
  	\dim_A\underline{M}\;:=\;\dim_L \,M\big/(M\cap\tau_M(\sigma^\ast M))\;-\;\dim_L \,\tau_M(\sigma^\ast M)\big/(M\cap\tau_M(\sigma^\ast M))\,.
  	\end{align*}
  \end{defi}
  A morphism of \Am s $ f:\underline{M} \to \underline{N}  $ is an $ A_L $-homomorphism $ f:M\to N $ satisfying $ f\circ \tau_M=\tau_N\circ \sigma^*f $. We denote by $ \Hom{\underline{M}}{\underline{N} }{L} $ the set of morphism.
  An \Am \  $ \underline{M} $ is called \emph{effective} if $ \tau_M $ comes from an $ A_L $-homomorphism $ \sigma^*M\to M $. An effective \Am \ has virtual dimension $ d\geq 0 $  { given by the dimension of  $ \mathrm{coker }\tau_M $}.\\
  With an \Am \ one can associate  $v$-adic Galois representations in the following way:
  Let $ A_v$ be the $ v$-adic completion of $ A $ at a closed point $v\neq \infty$ of $ \mathcal{C} $ and let $ Q_v $ be it's  fraction field. The \Am \ has an associated  \emph{$v$-adic \'{e}tale realization $\Heins{v}{M,A_v}$} (Tate-module) defined by
  \begin{align*}
  \Heins{v}{\underline{M},A_v}:=\check{T}_v\underline{M}:=({M}\otimes_{A_{L}} A_{v,L^{sep}})^\tau\\
  \Heins{v}{\underline{M},Q_v}:=\check{V}_v\underline{M}:=({M}\otimes_{A_{L}} A_{v,L^{sep}})^\tau \otimes_{A_v} Q_v;
  \end{align*} see definition~\ref{Def: Tate modules} and example~\ref{ex: associated shtuka} for the definition of $A_{v,L}$.
  Then Urs Hartl and Wansu Kim show in \cite[Prop. 4.2]{Hartl.Kim.2015} that this defines a contravariant functor to $\KatMod{A_v}$ and respectively $\KatVec{Q_v}$. Furthermore  $\Heins{v}{\underline{M},A_v}$ has a continuous Galois action by $ \AbsGalois{L}:= \Gal(L^{sep}/L)$, see \cite[Prop 6.1]{Taguchi.Wan.1996}. So, after a choice of $A_v$-basis for $\Heins{v}{M,A_v}$ we  can construct:
  \begin{align*}
  \GR{\underline{M},v}: \AbsGalois{L}\to \mathrm{Aut}_{A_v}(\check{T}_v\underline{M})\cong\mathrm{GL}_r({A_v})
  \end{align*}
  This factors through it's Zariski closure $H_{\underline{M}}$ in $A_v$, its \emph{'monodromy group'}. This group is also an algebraic group over $Q_v$, see for example~\cite[Lem. 4.2]{Hartl.Pal.2018}. The   image of the $v$-adic Galois representation $\GR{\M}$ attached to the $v$-adic \etale \ realization (Tate-module) can be a non-open in its Zariski closure. { Note that also in the number field case, a compact subgroup can be non-open in its Zariski closure if the closure is not semisimple, see~\cite{MO.CompactSubgroups}.}
  The aim of this paper is to give a criterion for this phenomena for rank one $A$-motives and conclude a  condition for the Galois openness of $A$-motives of higher rank.
  \subsection*{Overview}
  We start in the second and third section with a brief introduction to a local object associated to an \Am, a so-called \emph{local shtuka}. Here we  define Tate-modules and their associated Galois representation. Then section four deals with the Tannakian theory of the category of shtukas. We  also introduce the construction of Tannakian lattices and their associated Tannakian theory. In the last section, we formulate the main results of this paper. If $R$ is a valuation ring of a complete, discretely valued field, then they are as follows:
  \begin{heorem}\textbf{\normalfont{ \ref{prop: Openness rank one shtuka}.}}    Let $ \M $ be a local shtuka of rank $1$ and virtual dimension $d>0$ over $R$ and $K=\mathrm{Frac}(R)$. Let $ d'  \in \Z$ be the largest number such that $ d=p^e\cdot d' $ holds for $p=\mathrm{char}(\Fq)$, $ p \nshortmid d' $ and $e\geq 0$.
  	Then $\GR{\M}(\AbsGalois(K))$ is contained and open in $\mathrm{GL}_1(\F{v}\Laurrentreihe{z^{p^e}})$. In particular, $\GR{\M}(\AbsGalois{K})$ is open in $\HalgGp{\M}(Q_v)$ if and only if $p\nmid d$.
  \end{heorem}
  \begin{rop}\textbf{\normalfont{ \ref{prop: openes for higher r}.}} Let $\M$ be a local shtuka of rank $r\geq1$  and dimension $d$ over $R$.\\
  	If $\mathrm{char}(\F{q}) | d$ then ${\GR{\M}(\AbsGalois{K})}\subseteq \mathrm{GL}_r(\F{v}\Laurrentreihe{z})$ is not open.
  \end{rop}
  \begin{orollary}\textbf{\normalfont{ \ref{Cor: Amot is not open}.}}
    Let $ \underline{M} $ be a  \Am \ of rank $r$  and  dimension $d$ over a finite field extension $L$ over $Q$.
	  If $\mathrm{char}(\F{v}) | d$ then ${\GR{\M}(\AbsGalois{K})}\subseteq \mathrm{GL}_r(\F{v}\Laurrentreihe{z})$ is not open.
  \end{orollary}

 \section{Local shtukas}
    In this section we introduce local shtukas which can be attached to \Am s. When we define the Tate module of an \Am \ this structure is somewhat in between. However, in contrast to the Tate module, it can be defined for all places of $\MC$. So we will first give a general introduction to local shtukas and will explain their connection to \Am s in example~\ref{ex: associated shtuka}.\\
    Let $A$ and $Q$ be  as in the introduction. Let $K$ be a field, which is complete with respect to a non-trivial and non-archimedian absolute value $|\cdot|\!:K\to\R_{\geq 0}$. Assume further that  its value group $|K^\times|$ is discrete in $\R_{\geq 0}$ and let $K^{alg}$ and $K^{sep}$ a fixed algebraic and separable closure of $K$.
    Let $R=\{x \in K: |x|\leq 1 \}$ be the valuation ring of $K$, $\mathfrak{m}_R=\{x \in K: |x|< 1 \}$ be the maximal ideal of $R$ and $k:= R/\mathfrak{m}_R$ its residue field. Let $\gamma:A \to R$ be  an injective ring morphism from $A$ to $R$.  It induces a homomorphism $\ A \to k$. Denote by $v= \gamma^{-1}(\mathfrak{m}_R) $ its induced kernel. We assume that $v$ is a maximal ideal.
    For example $K$ can be the completion of a finite field extension $L$ of $Q$ at a place above $v$.\\
    Denote by $\F{v}:=A/v$  the residue field of $A$ at $v$ and set $q_v:= \# \F{v}$. We fix a uniformizing parameter $z:= z_v $ at $ \mathcal{O}_{\mathcal{C},v} $. Then there exists a canonical isomorphism between the completion of $\mathcal{O}_{\mathcal{C},v}$,  denoted by $A_v$, and   $\F{v}\Laurrentreihe{z} $, see \cite[Lem. 1.2 \& 1.3]{Hartl.Juschka.2016}.
    The canonical isomorphism also implies   $Q_v\cong\F{v}\DoppelRund{z}$. The homomorphism $\gamma$  extends by continuity  to a homomorphism  $A_v\to R$ and set  $\zeta:= \gamma(z)$. \\
    Let $\Rz$ be the ring of formal power series in the variable $z$ over $R$. We set $ \hat{\sigma}:= \sigma^{[\F{v}:\Fq]}$ as the endomorphism with $\hat{\sigma}(z)=z$ and $ \hat{\sigma}(a)= a^{q_v}$ for $a \in R$.
    For a $\Rz$-module $M$ set $\hat{\sigma}^*M:= M\otimes_{\Rz,\hat{\sigma}}\Rz$, $M[{\frac{1}{z-\zeta}}]:= M\otimes_{\Rz}\Rz [{\frac{1}{z-\zeta}}]$  and also $M[{\frac{1}{z}}]:= M\otimes_{\Rz}\Rz [{\frac{1}{z}}]$.
    \begin{defi}\label{def: shtukas}
    	A \emph{local $\hat{\sigma}$-shtuka (or local shtuka)} over $R$ of rank $r$ is a pair $\M=(\hat{M},\tau_{\hat{M}})$ consisting of a  free $\Rz$-module $\hat{M}$ of rank $r$ and an isomorphism
    	\begin{align*}
    	\tau_{\hat{M}}:\hat{\sigma}^*\hat{M}\left[ {\textstyle \frac{1}{z-\zeta}} \right]\xlongrightarrow{\sim} \hat{M}\left[{\textstyle\frac{1}{z-\zeta}}\right].
    	\end{align*}
    	A \emph{morphism} of local shtukas $f: \M=(\hat{M},\tau_{\hat{M}})\to \N=(\hat{N},\tau_{\hat{N}})$ over $R$ is a morphism of $\Rz$-modules $f:\hat{M}\to \hat{N}$ satisfying $ \tau_{\hat{N}} \circ \hat{\sigma}^*f=f\circ \tau_{\hat{M}}$. We denote the set of morphism by $\Hom{\M}{\N}{R}$. The category of local shtukas over $R$ is denoted $\KatLshtuka{A_v,R}$\\
    	A \emph{quasi-morphism} between local shtukas $f: \M=(\hat{M},\tau_{\hat{M}})\to \N=(\hat{N},\tau_{\hat{N}})$ over $R$ is a morphism of $R\DoppelRund{z}$-modules
    	$f:\hat{M}[\frac{1}{z}]\to \hat{N} [\frac{1}{z}]$ satisfying $ \tau_{\hat{N}} \circ \hat{\sigma}^* f=f\circ \tau_{\hat{M}}$.
    	A quasi-morphism is called \emph{quasi-isogeny}  if it is an isomorphism of $R\DoppelRund{z}$-modules.
    	We denote the set of quasi-morphism by $\QHom{\M}{\N}{S}$. The category with  local shtukas as objects and $\QHom{\M}{\N}{S}$ as Hom-set is denoted by $\KatLshtuka{Q_v,R}$.\\
    	If $\tau_{\hat{M}}(\hat{\sigma}^*\hat{M})\subseteq \hat{M}$ then $\M$ is called \emph{effective}. If furthermore  $\tau_{\hat{M}}(\hat{\sigma}^*{\hat{M}})=\hat{M}$ holds, then $\M$  is called \emph{\etale}.
    \end{defi}
    By \cite[Cor. 4.5]{Hartl.Kim.2015} $\Hom{\M}{\N}{S}$ is a finite free $A_v$-module of rank at most $\rk \M \cdot \rk \N$. Furthermore, $ \QHom{\M}{\N}{S}=\Hom{\M}{\N}{S}\otimes_{A_v} Q_v$ holds.
    \begin{Lemma}(\cite[Lem. 2.3.]{Hartl.Kim.2015})\label{Lemma: dimension}
    	Let $\M=(\hat{M},\tau_{\hat{M}})$ be a local shtuka of rank $r$ over  $R$.
    	\begin{enumerate}
    		\item There is an integer $d\in \Z$ such that $\mathrm{det}\tau_{\hat{M}} \in (z-\zeta)^d R\Laurrentreihe{z}$
    		\item If $\M$ is effective the integer from (1) satisfies $d\geq 0$ and $\hat{M}/\tau_{\hat{M}}(\hat{\sigma}^*\hat{M})$ is a free $R$-module of rank $d$ which is annihilated by $(z-\zeta)^d$.
    	\end{enumerate}
    \end{Lemma}
    \begin{defi}
    	The integer $d$ defined by the above lemma is called the \emph{dimension of  $\M$} and denoted by $\mathrm{dim }\M$.
    \end{defi}
    \begin{Corollary}\label{Cor:Form local shtuka}
    	Every  local shtuka $\M$ over $R$ is isomorphism to the tensor product of a power of the Carlitz shtuka over $R$ and an effective local shtuka over $R$. In the case of rank $1$ local shtukas, we have $\M=(R\Laurrentreihe{z},c \cdot (z-\zeta)^d)$ for a $c \in R\Laurrentreihe{z}^\times$.
    \end{Corollary}
    \begin{proof}
    	With Lemma~\ref{Lemma: dimension} every local shtuka $ \M \in \KatLshtuka{\Fqrz} $ is effective after tensoring it with $\mathbb{1}(n)$ for some $ n\in \Z $, i.e.   $ \M \otimes \mathbb{1}(n)=: \M' $ is effective. Thus $ \M\cong \M'\otimes \mathbb{1}(-n) $.
    \end{proof}
    A rich class of examples is given by the  category of \Am s in the following way:
    \begin{ex}\label{ex: associated shtuka}
    	Let $L$ be a finite extension of $Q$ as in the introduction, let $K$ be  a completion of $L$ at a place $w$ lying above $A$ and let $R$ be a valuation ring of $K$.  An \Am \ $\underline{\mathcal{M}}$ over $R$  is a pair $(\mathcal{M},\tau_{\mathcal{M}})$, consisting of a locally free $A_R:= A \otimes_{\Fq} R$-module $\mathcal{M}$ of finite rank and an isomorphism $   \tau_{\mathcal{M}}: \sigma^*\mathcal{M}\lvert_{\Spec A_R \backslash V(J)} \to \mathcal{M}\lvert_{\Spec A_R \backslash V(J)} $
    	of the associated sheaves outside the vanishing locus of the ideal $J:=(a\otimes1-1\otimes\gamma(a)|a\in A )= \ker(\gamma \otimes id_R:A_R\to R)$.\\
    	An \Am \ $\underline{M}$ over $K$ has  a good model over $R$, if there exists an \Am \ $\mathcal{\underline{M}}$ over $R$ satisfying $\mathcal{\underline{M} }\otimes_R K\cong\underline{M}$. Its reduction $\mathcal{\underline{M} } \otimes_R R/\mathfrak{m}_R$ is an \Am \ over $R/ \mathfrak{m}_R$.
    	Note that the \Am \ $\underline{M} $ over $L$ always has generic A-characteristic and the reduction has A-characteristic $v=\ker(A\to \mathfrak{m}_R)$.
    	We denote by $A_{v,R}$ the completion of $\MO{\MC_R}$ at the closed subscheme $v \times \Spec R$ and  by $ Q_{v,R}=A_{v,R}[\frac{1}{v}]$ its fraction field. We do not assume that $ v \times \Spec R$ consists of one point.
    	Thus, to give the associated shtuka to an \Am \ over $R$, we are required to take $f_v:= [\F{v}:\Fq] \geq 1$ into considerations. If $f_v=1$ then there is only one point over $v$ and the fields $\F{v} \cong \Fq$ are isomorphic, with $q_v=q$, $ \hat{\sigma}=\sigma^{q_v} $ and hence $ A_{v,R}=R\Laurrentreihe{z}$ holds.
    	Thus we associate with $\mathcal{\underline{M}}$ the local shtuka defined by $\mathrm{\underline{\hat{M}}}_v(\underline{\mathcal{M}}):=\mathcal{\underline{{M}}} \otimes_{A_R} R\Laurrentreihe{z}=(\mathcal{M} \otimes_A R\Laurrentreihe{z}, \tau_M \otimes \mathrm{id_{R\Laurrentreihe{z}}})$ over $R$.\\
    	Now let $f_v > 1$, then it gets more complicated as  $A_{v,R}$ is not an integral domain; compare  \cite[after Prop. 4.8]{Bornhofen.Hartl.2011}. Consider the extension  $\gamma:A_v\to R$  and define for $i\in \Z/f_v\Z$  the ideals $\mathfrak{a}_i:= (a\otimes 1-1\otimes \gamma(a)^{q^i}|a\in \F{v}) $ in $ A_{v,R} $.
    	As $\F{v}$ is a finite field extension of $\Fq$, it is given by some element $a\in\F{v}\setminus \F{q}$ with $\F{v}=\F{q}(a)$ and minimal polynomial $m_{a}$. Thus $m_a$ divides the polynomial $\prod_{i\in \Z/f_v\Z} (X-a^{q^i})$, in which each factor represents an  $\mathfrak{a}_i$ and  hence $\prod_{i\in \Z/f_v\Z}\mathfrak{a}_i=0$.
    	By the Chinese remainder theorem we have the decomposition
    	\begin{align*}
    	A_{v,R}= \prod_{\Gal({\F{v}/\Fq})} A_v \hat{\otimes}_{\F{v}} R = \prod_{i \in  \Z/{f_v}\Z} A_{v,R} / \mathfrak{a}_i
    	\end{align*}
    	and a similar decomposition of $Q_{v,R}$. The ideals $\mathfrak{a}_i$ corresponds to the places $w_i$ of $\MO{\MC,\F{v}}$ lying above $v$. Each factor of the decomposition is canonically isomorphic to $\Rz$ and is permuted by $\sigma $ as $ \sigma(\mathfrak{a}_i)= \mathfrak{a}_{i+1}$ holds.
    	In particular $\sigma^{f_v}$ fixes each factor. Under the above decomposition  $J$ splits up as $ J \cdot A_{v,R} / \mathfrak{a}_0 =(z-\zeta)$ and $ J \cdot A_{v,R} / \mathfrak{a}_i=(1) $ for $i\neq 0$. Thus we associate with $\mathcal{\underline{M}}$ the local shtuka defined  by
    	$\underline{\mathrm{\hat{M}}}_v(\mathcal{\underline{M}}):= \mathcal{\underline{M}}  \otimes_{A_R} A_{v,R} / \mathfrak{a}_0= (M \otimes_{A_R} A_{v,R}/ \mathfrak{a}_0,(\tau_M\otimes id )^{f_v})$
    	where $\tau_M^{f_v}:= \tau_M \circ \sigma^*\tau_M\circ \dots \circ {\sigma^*}^{f_v-1}\tau_M $.
    \end{ex}
    \begin{defi}\label{def: associated shtuka}
    	The \emph{local $\hat{\sigma}$-shtuka associated at $v$} to an  \Am \ $\underline{M}$ over a valuation ring $R$  is defined as
    	\begin{align*}
    	\underline{\mathrm{\hat{M}}}_v(\underline{M}):= \underline{M} \otimes_{A_R} A_{v,R}= (M \otimes_{A_R} A_{v,R},(\tau_M\otimes id_{A_{v,R}} )^{f_v})
    	\end{align*}
    	where $f_v=[\F{v}:\Fq]$ and $\tau_M^{f_v}:= \tau_M \circ \sigma^*\tau_M\circ \dots \circ {\sigma^*}^{f_v-1}\tau_M $.
    \end{defi}
    The local shtuka associated at $v$ to an \Am  \  in the above way preserves effectiveness and \etale ness, see \cite[Ex. 2.2]{Hartl.Kim.2015} and
    has rank $\mathrm{rk} \ \underline{\mathrm{\hat{M}}}_v(\underline{M})=\mathrm{rk}\ \underline{M}$ and  dimension $\mathrm{dim} \  \underline{\mathrm{\hat{M}}}_v(\underline{M})= \mathrm{dim}\ \underline{M}$.
    \begin{ex}\label{ex: Carlitz introduction}
    	Let $\MC =\mathbb{P}^1_{\Fq}$, $A= \Fqt$, $Q=\Fqrt$ and let $L=\Fq(\vartheta)$. Then $\gamma: \Fqt \to \Fq(\vartheta)$ is given by $t\mapsto \vartheta$ and induces an isomorphism $Q \cong L$. The Carlitz \Am \ is defined as the \Am
    	\begin{align*}
    	\underline{C}=(L[t], (t-\vartheta))
    	\end{align*} over $L$. Consider the ideal $v=(z(t))$ generated by a monic irreducible polynomial $z(t) \in A$ of arbitrary degree. The ideal $v$ is maximal  and set  $\F{v}:= A/(z)$. Let $A_v$ be the completion $\MO{\MC,v}$  at $v$ and hence there exists the canonical isomorphism $\F{v}\Laurrentreihe{z}\xlongrightarrow{\sim} A_v$.\\
    	Let $\zeta$ be the image of $z$ under $\gamma$  and consider the valuation ring $R=\F{v}\Laurrentreihe{\zeta}$. Denote again by $\vartheta$ the image of $t$. The Carlitz \Am \ $\underline{C}$ has a good model over $R$ with good reduction given by the \Am \ $ \underline{C}=(R[t], (t-\vartheta))$ over $R$. The  associated local shtuka  is given by
    	\begin{align}
    	\underline{\hat{C}}:= \mathrm{\M}_v(\underline{C})= (\Rz, z-\zeta)
    	\end{align} regardless the degree of $z(t)$, see \cite[Ex. 2.7]{Hartl.Kim.2015}.
    \end{ex}
    \begin{rem}
    	Note that the Carlitz shtuka defined above is exactly the Tate object in the category $\KatLshtuka{A_v,R}$ and  $\KatLshtuka{Q_v,R}$. Therefore, we will denote it by $\mathbb{1}(1):= \underline{\hat{C}}$. It has rank and dimension $1$. A tensor power is given by $\mathbb{1}(n):= \underline{\hat{C}}^{\otimes n}:= (\Rz, (z-\zeta)^n)$ and has the same rank as $\underline{\hat{C}}$ but dimension $n$.
    \end{rem}

\section{Tate modules and Galois representation}\label{sec: Tate modules and Galois representation}
We keep the notation from the last section and let  $\M$ be a local shtuka over   $R$. This implies that $\tau_{\hat{M}}$ is an isomorphism over the extension to $K\Laurrentreihe{z}$. So $\M\otimes_{\Rz} K\Laurrentreihe{z}$ is an \etale \ local shtuka and we have the following definition:
\begin{defi}\label{Def: Tate modules}
	The   \emph{(dual) Tate module}   of a local shtuka $\M $ over $R$ is defined as
	\begin{align*}
	\TateTFunc{\M}:=(\M\otimes_{\Rz} K^{sep}\Laurrentreihe{z})^{\tau}:=\{ m\in \M\otimes_{\Rz} K^{sep}\Laurrentreihe{z}|\tau_{\hat{M}}(\hat{\sigma}^*(m)=m)\}
	\end{align*}
	and the \emph{rational Tate module} is defined as
	\begin{align*}
	\TateVFunc{\M}:= \{ m\in \M\otimes_{\Rz} K^{sep}\DoppelRund{z}|\tau_{\hat{M}}(\hat{\sigma}^*(m)=m)\}=\TateTFunc{\M}\otimes_{A_v}Q_v
	\end{align*}
\end{defi}
The name dual comes from the following reason. If ${\M}= {\M}(\mathcal{\underline{M}})$ comes from an \Am \ $\mathcal{\underline{M}}$ as in example~\ref{ex: associated shtuka} and
if $\underline{\mathcal{M}}=\mathcal{\underline{M}}(\varphi)$ is the \Am \ associated to a Drinfeld-module $\varphi$, then $\TateTFunc{\M}$  is dual to the Tate-module $T_v\varphi:=\varprojlim \varphi[v^n]$ of $\varphi$, see~\cite[Prop 4.8]{Hartl.Kim.2015}.\\
They are also called the \emph{ the $v$-adic realization of $\M$} and thus are denoted by  $\Heins{v}{\M,A_v}:=\TateTFunc{\M}$.\\
The Tate-module of the associated local shtuka to an \Am \  and the Tate-module of the \Am \ are isomorphic, see \cite[Prop. 4.6]{Hartl.Kim.2015}. However the former has a continuous $\AbsGalois{K}$- and the latter a continuous $\AbsGalois{L}$-action.
It is also shown in \cite[Prop. 4.2]{Hartl.Kim.2015} that $\TateTFunc{\M}$ is a free $A_v$-module of rank equal to $\rk\M$ and that $\TateVFunc{\M}$ is a $Q_v$-vector space of dimension also equal to $\rk\M$. Then Hartl and Kim shows that $\M \otimes_{\Rz} K\Laurrentreihe{z}$ can be recovered by the Galois invariants $(\TateTFunc{\M}\otimes_{A_v} K^{sep}\Laurrentreihe{z})^{\AbsGalois{K}}$ via the canonical isomorphism
\begin{align}\label{eq:HK isom}
\TateTFunc{\M}\otimes_{A_v} K^{sep}\Laurrentreihe{z} \xlongrightarrow{\sim} \M\otimes_{\Rz} K^{sep}\Laurrentreihe{z}.
\end{align}
To give an $A_v$-basis for the (rational) Tate module of $ \M:=(\hat{M},\tau_{\hat{M}}) $ over  $  R $, set $ r=\mathrm{rk}\ \M $. We choose a $\Rz$-basis of $\M$ and denote by $T$ the structure homomorphism $\tau_{\hat{M}}$ with respect to this basis.\\
Note that $ T\in \mathrm{GL}_r(\Rz\left[ \frac{1}{z-\zeta}\right]) $. Then a $ \F{v}\Laurrentreihe{z} $-basis $U \in GL_r(K^{sep}\Laurrentreihe{z})$ for the Tate module has to satisfy by definition $T\hat{\sigma}(U)=U$.
In the proof of \cite[Proposition 4.2 \& Remark 4.3]{Hartl.Kim.2015} the authors shows that $U$ is already in $ K\Erzeugt{\frac{z}{\zeta^{q}}}$. Then by \cite[Lemma 5.22]{Hartl.Kim.2015} one can solve for an effective local shtuka the equation $\hat{\sigma}(X)= X \cdot T $  for an $X \in \MO{K^{sep}}\Laurrentreihe{z}^{r\times r}$ and hence  $U$ is the inverse of the solution.
\begin{ex}\label{ex: Carlitz tate module}
	We continue with the Carlitz shtuka $\underline{\hat{C}}=(\Rz, (z-\zeta))$ from example~\ref{ex: Carlitz introduction}.
  The rank of $\underline{\hat{C}}$ is $ r=1 $, so after choosing a $\Rz$-basis for $\underline{\hat{C}}$,  we get $ T=\tau_{C,\Rz} =(z-\zeta)$.
  Let $K=\mathrm{Frac}(R)$.
	Then an inverse element  $ x=\sum_{i=0}^{\infty} x_i z^i \in GL_1(K^{sep}\Laurrentreihe{z}) $ of the Tate-module basis has to satisfy
	\begin{align}\label{eq: Carlitz-Tate equation}
	x_0^{q-1}&=-\zeta \nonumber \\
	x_i^{q} +\zeta \ x_i&=x_{i-1} \text{ for } i > 0
	\end{align}
	In what follows these equations are considered as \emph{Carlitz-Tate equations}. Let $ l_i  \in K^{sep}$ be the solutions of the  Carlitz-Tate equations for $i\geq 0$ and set $ l_+:=\sum_{i=0}^{\infty} l_i z^i \in \MO{K^{sep}}\Laurrentreihe{z}$. The element satisfies $  \sigma(l_+)=(z-\zeta)l_+$.
	Then a  basis $ u $ for $ \TateTFunc{\underline{\hat{C}}} $ is defined by $ u=l_+^{-1} $ which implies $ \TateTFunc{\Carlitz}=l_+^{-1}A_v $ and $\TateVFunc{\underline{C}}= l_+^{-1}Q_v$.
	The basis depends on the choice of $ l_i $ and a different choice leads to a different power series which is a multiple of $ l_+ $ and  an unit of $  (K^{sep}\Laurrentreihe{z})^{\sigma=id}=A_v^\times$, cf. \cite[Ex. 4.10]{Hartl.Kim.2015}.
	The canonical isomorphism~\eqref{eq:HK isom} is given by sending the generator $l_+^{-1}$ of $\TateTFunc{\M}$ to the same element in $(K^{sep})^\times$.
\end{ex}
We  define the covariant functors induced by the Tate modules as
\begin{align*}
\underline{\TateTFunc{}}:\  &\KatLshtuka{A_v,R}\to \KatRep{A_v}{\AbsGalois{K}}{cont}, &\M \mapsto (\rho:\AbsGalois{K}\to GL(\TateTFunc{\M})),\\
\underline{\TateVFunc{}}: \ &\KatLshtuka{Q_v,R}\to \KatRep{Q_v}{\AbsGalois{K}}{cont}, &\M \mapsto (\rho:\AbsGalois{K}\to GL(\TateVFunc{\M})).
\end{align*}
The category  $ \KatRep{Q_v}{\AbsGalois{K}}{cont} $ of continuous $ \AbsGalois{K} $-representation is a neutral Tannakian category with fiber functor $ \omega_{rep}:\KatRep{Q_v}{\AbsGalois{K}}{cont} \to \KatVec{Q_v}$ given by $(\rho:\AbsGalois{K}\to GL(V))\mapsto V $.
Let $\M$ be a local shtuka, then we can associate a Tannakian subcategory $\DoppelSpitz{\underline{\TateVFunc{\M}}}$ of  $\KatRep{Q_v}{\AbsGalois{K}}{cont} $.
{Note that it makes no difference in which shtuka category $\M$ lives, as the natural transformation from $\underline{\TateTFunc{}}$ and $\underline{\TateVFunc{}}$ is by definition compatible with the tensor structure.}
By Tannakian duality there exists a linear affine algebraic group $ \HalgGp{\M}:= \TensorAut{\omega_{rep}}{\DoppelSpitz{\underline{\TateVFunc{\M}}}}$ over $ Q_v $ such that the category $\DoppelSpitz{\underline{\TateVFunc{\M}}}$ is tensor equivalent to the category of $Q_v$-representations of $\HalgGp{\M}$. We call the group $\HalgGp{\M}$  \emph{monodromy group of $\M$}. The monodromy group of $\M$ is a subgroup in $ \mathrm{GL}(\TateVFunc{\M}) $ or more to the point in the centralizer of $\QEnd{\M}$ in $ GL(\TateVFunc{\M}) $.\\
For all $ g \in \AbsGalois{K} $, $ \rho(g)$ is already an automorphism of $ {\TateVFunc{\M}} $ over $Q_v$. Moreover they are already automorphism on  the subcategory generated by $\TateVFunc{\M}$. Then $ \rho(g) $ is an element of $ \TensorAut{\omega_{rep}}{\DoppelSpitz{\underline{\TateVFunc{\M}}}} $ and so in $ \HalgGp{\M} $.
\begin{defi}
	We define the \emph{Galois representation of a local shtuka $\M$ of rank $r$} as the representation
	\begin{align*}
	\GR{\M}:=\GR{\M,v}:\AbsGalois{K} &\to \HalgGp{\M} (Q_v).
	\end{align*}
\end{defi}
If we choose a basis for a local shtuka $\M$ and a basis $U$ for its Tate-module $\TateTFunc{\M}$, then it also follows by \cite[Remark 4.3]{Hartl.Kim.2015} that $U \in \mathrm{GL}_r(K{sep}\Laurrentreihe{z})$.
Thus $U$ is a basis transformation matrix for $\HalgGp{\M}(Q_v)$ with $U^{-1}g(U)\in \mathrm{GL}_r(A_v)$ for a $g\in \AbsGalois{K}$. Then  $\GR{\M}$ is given by $g\mapsto U^{-1}g(U)$ for a $g\in \AbsGalois{K}$.\\
By \cite[Lem. 4.2]{Hartl.Pal.2018} the algebraic group $ \HalgGp{\M} $ is the Zariski closure of $ \GR{\M}(\AbsGalois{K})$ in $ \mathrm{GL}_{Q_v}(\TateVFunc{\M}) $.
\begin{ex}\label{Ex: Carlitz representation}
	Continue with example~\ref{ex: Carlitz tate module} of the Carlitz shtuka $\underline{\hat{C}}$.
	We have chosen an uniformizer $z$ and hence have $A_v\cong\F{v}\Laurrentreihe{z}$ and $K=\F{v}\DoppelRund{\zeta}$ as $R$ contains $\F{v}\Laurrentreihe{\zeta}$. Then set $ K_n:=\F{v}\DoppelRund{\zeta}(l_0, \dots, l_{n-1}) $ and $ K_\infty:= \F{v}\DoppelRund{\zeta}(l_i:i\in \mathbb{N}_0) $ for the field extension of $K$ by the solutions of the Carlitz-Tate equations.
	This field tower is  the function field analogous of the cyclotomic tower $\Q(\sqrt[p^i]{1}| i \in \N_0)$ and we call it the Carlitz-Tate towers. Then  we get the so called $v$-adic cyclotomic character
	\begin{align}\label{eq: Carlitz rep isom}
	\GR{\Carlitz}: \Gal(K_\infty/K) &\xlongrightarrow{\sim}\F{v}\Laurrentreihe{z}^\times=\GM(\F{v}\Laurrentreihe{z})=\Aut{\F{v}\Laurrentreihe{z}}{\TateTFunc{\M}}\\
	g &\longmapsto l_+g((l_+)^{-1})= \GR{\C}(g), \nonumber
	\end{align}
	which satisfies $ g(l_+)^{-1}= \GR{\Carlitz}(g)\cdot l_+^{-1}$ for all $g \in  \Gal(K_\infty/K)$.  Note that, as $ \Carlitz $ has rank $ 1 $, $\GM{}$ is the only canonical choice for $ \HalgGp{\Carlitz} $.\\
	The reason for the isomorphism~\eqref{eq: Carlitz rep isom} is that $K_n$ is totally ramified over $K$ of degree $(q-1)q^n$, see~\cite[Example 4.10]{Hartl.Kim.2015}.
	In particular, if $K^{un}=\overline{\F{v}}\DoppelRund{\zeta}$ is the completion of the maximal unramified field extension, then still $K_n\cdot K^{un}$ is totally ramified over $K^{un}$ of degree  $(q-1)q^n$. So the isomorphism~\eqref{eq: Carlitz rep isom} also holds for $K^{un}$ instead of $K$.
\end{ex}

\section{Tannakian theory on $\KatLshtuka{}$}
The categories $\KatLshtuka{A_v,R}$ and $\KatLshtuka{Q_v,R}$ are  additive, have internal homs,  a tensor structure and  are rigid, which all descents from their $ \Rz $-module structure. To be more precise, let $ \M=(\hat{M},\tau_{\hat{M}}) $ and $ \N=(\hat{N},\tau_{\hat{N}}) $ be two local shtukas over $ R $. The internal  Homs $ \operatorname{\underline{Hom}}(\M,\N) $ are defined for the modules $ \hat{M} $ and $ \hat{N} $ as the usual  $ \Hom{\hat{M}}{\hat{N}}{\Rz} $ and the shtuka structure is given by  $ h\mapsto \tau_{\hat{N}} \circ h \circ \tau_{\hat{M}}^{-1} $ for a $ h \in \sigma^*\operatorname{\underline{Hom}}(\M,\N) $.
In the same way is the tensor product given by the tensor product of the modules $\hat{M}\otimes_{\Rz} \hat{N}$ and the shtuka structure by the tensor product $\tau_{\hat{M}} \otimes \tau_{\hat{N}}$. The unit element associated to the tensor structure is the shtuka $ \mathbb{1}(0):= (\Rz, id_{\Rz}) $.
We  define a tensor power of an local shtuka $ \M $ by $ \M^{\otimes 0}:= \mathbb{1}(0) $ and $ \M^{\otimes n}:= \M^{\otimes n-1}\otimes \M $ for $ n \in \mathbb{N} $.
The dual  $ \M^\vee:=\operatorname{\underline{Hom}}(\M,\mathbb{1}(0)) $ consists of the module $ \hat{M}^\vee:= \Hom{\hat{M}}{\Rz}{\Rz}  $  and the isomorphism $ (\tau_{\hat{M}})^{-1} $.
\begin{prop}\label{Prop: Sht is Rigid additive}
	The categories $ \KatLshtuka{A_v,R} $ and  $ \KatLshtuka{Q_v,R} $ are rigid additive tensor category over $A_v$ and respectively  $Q_v$. Furthermore  $ \KatLshtuka{Q_v,R} $ is abelian.
\end{prop}
\begin{proof}
	By additivity and \cite[1.15]{Deligne.Tannakian.1989} is $\KatLshtuka{A_v,R}$  a $ A_v $-linear and  $\KatLshtuka{Q_v,R}$ a  $ Q_v$-linear categories.
  The endomorphism ring  of $\KatLshtuka{A_v,R}$ is given by $ \EndCat{\KatLshtuka{A_v,R}}{\mathbb{1}(0)}= A_v$ and of $\KatLshtuka{Q_v,R}$  by  $ \EndCat{\KatLshtuka{Q_v,R}}{\mathbb{1}(0)}= Q_v$.\\
	Let $ \M,\N \in \KatLshtuka{A_v,R} $ and $ f:\M \to \N $. Consider the kernel and image of $ f $ in the category of free modules $ \KatMod{\Rz}^f $.
  As  submodules of $ \M $ and $ \N $ they  have a shtuka-structure. The cokernel in $ \KatMod{\Rz}^f $ is given by the usual cokernel modulo torsion. It also inherits the shtuka structure of $\N  $ as submodule.\\
	However the natural  morphism $ \operatorname{Coim} f \to \mathrm{Im} f$ is only an isogeny. Thus only an isomorphism in $ \KatLshtuka{Q_v,R} $, which was to show.
\end{proof}
To relate the abelian theory of $\KatLshtuka{Q_v,R}$ to  $\KatLshtuka{A_v,R}$, we introduce so called \emph{abelian and Tannakian lattices}; compare \cite{Wedhorn.2004} and \cite{Duong.Hai.2013}.
\begin{defi}\label{Def: Wedhorn Extension}
	Let $ t:S\to S' $ be a ring homomorphism  of Dedekind rings and let $ C $ be an $S$-linear category. We define the \emph{local scalar extension} $ C $ along $ t $  as the category $ C_{S'} $ with the same  object class as $ C $ and $ hom$-set of two objects $ X,Y \in C_{S'} $ given by
	\begin{align*}
	\Hom{X}{Y}{C_{S'}}:= \Hom{X}{Y}{C}\otimes_R S'
	\end{align*} Thus $ C_{S'} $ is an $ S' $-linear category with a $S  $-linear functor $ C\to C_{S'} $.\\
	We define the \emph{local fraction extension of $C$}  as the extension $ C_{\mathrm{Frac}(S)}$.
\end{defi}
Obviously the extension the functor $ C\to C_{S'} $ is essential surjective. If $ t $ is flat, it preserves epi- and monomorphism and the converse holds if $ t$ is faithfully flat, see~\cite[\S 3.6]{Wedhorn.2004}. If $ C $ is a monoidal category, then $ C_S $ also, see~\cite[\S 3.8]{Wedhorn.2004}. A $ S $-linear  functor $ s:C\to D $ induces a $ S' $-linear functor $ s_{S'}: C_{S'}\to D_{S'} $.
\begin{defi}
	An \emph{abelian lattice $ \Lambda$ over a Dedekind ring $ S $} is a  $ S $-linear additive symmetric monoidal category  $\Lambda$ with an isomorphism $ S \cong \End{}{\mathbb{1}_\Lambda} $ and in which kernels and images exist, such that the local fraction extension $ \Lambda_{\mathrm{Frac}(S)} $ is a rigid abelian monoidal category over $  \mathrm{Frac}(S) $.
\end{defi}
The existence of images is in the sense that for each morphism $f:X\to Y$ in $C$ an image factorization exists. That is the most general definition of an image and follows in the case of local shtukas from the existence of the cokernel.
\begin{Corollary}
	$ \KatLshtuka{A_v,R} $  is  a rigid  abelian  tensor lattice over $A_v$ with local fraction extension $ (\KatLshtuka{A_v,R})_{Q_v}=\KatLshtuka{Q_v,R} $.
\end{Corollary}
\begin{proof}
	This follows by definition~\ref{def: shtukas} and proposition~\ref{Prop: Sht is Rigid additive}.
\end{proof}
\begin{defi}\cite[2.2.2]{Duong.Hai.2013}
	Let $S$ be a Dedekind ring.	A \emph{neutral Tannakian lattice $ \Lambda $ over $  S$}  is a rigid abelian lattice $\Lambda$ equipped with a $ S$-linear additive tensor functor $ \omega: \Lambda \to \KatMod{S} $, such that $ \omega$ is faithful, preserves kernels and images and is fully faithful restricted to the monoidal subcategory generated by the unit element.
\end{defi}
\begin{prop}
	The category $ \KatLshtuka{Q_v,R} $ is a neutral Tannakian category with fiber functor $  \omega_{Sht_Q}:= \TateVFunc{}: \KatLshtuka{Q_v,R}\to \KatVec{Q_v}$ and the categroy $ \KatLshtuka{A_v,R} $ is a neutral Tannakian lattice with fiber functor $ \omega_{Sht_A}:= \TateTFunc{}: \KatLshtuka{A_v,R} \to \KatMod{A_v} $.
\end{prop}
\begin{proof}
	With \cite[Theorem 4.20]{Hartl.Kim.2015}  both functors are fully faithful and hence fiber functors in the sense described above.
\end{proof}
By Tannakian duality one can associate to  $\M \in \KatLshtuka{Q_v,R}  $  an affine linear algebraic group $ \MG{\M}:=\TensorAut{\omega_{Sht_Q}}{\DoppelSpitz{\M}} \subseteq \GL(\TateVFunc{\M}) $.
\begin{defi}
	$ \MG{\M}  $ is called the \emph{Motivic Galois group} of $ \M $.
\end{defi}
$\KatLshtuka{Q_v,R}$ is a Tannakian category and $\underline{\TateVFunc{}}$ is an exact tensor functor over $\KatRep{Q_v}{\AbsGalois{K}}{cont}$. So we have that the monodromy group $\HalgGp{\M}$ of a local shtuka $\M$ is a subgroup of the motivic Galois group $\MG{\M}$.\\
In \cite{Duong.Hai.2013} Nguyen Dai Duong and Ph\`{u}ng H\^{o} Hai established a similar theory for Tannakian lattices.  Then main theorem  is as follows
\begin{Theorem} \cite[2.3.2]{Duong.Hai.2013}
	Let $ (\Lambda,\omega) $ be a Tannakian lattice over a Dedekind ring $ S $. Then the groups scheme $ {\iMG{\M}}{}:= \TensorAut{\omega}{\Lambda} $ is faithfully flat over $ S $ and $ \omega $ induces a equivalence between $ \Lambda $ and the category of finite projective representations $ \KatRep{S}{{\iMG{\M}}{}}{fp} $.
\end{Theorem}
We call  the algebraic group $ {\iMG{\M}}{}:= \TensorAut{\omega_{Sht_A}}{\DoppelSpitz{\M}} $ associated to a  $\M \in \KatLshtuka{A_v,R} $ the \emph{integral motivic Galois group of $\M$}. Moreover, ${\iMG{\M}}{}$ is after  extension from $A_v$ to $Q_v$ the motivic Galois group of $\M$. Thus, going back to the categories $\KatLshtuka{A_v,R}$ and $\KatLshtuka{Q_v,R}$, the following diagram {commutes}
\begin{align*}
\begin{xy}
\xymatrix{ \KatLshtuka{A_v,R} \ar[d]_-{\TateTFunc{}} \ar[r]^-{}&  \KatLshtuka{Q_v,R}  \ar[d]_-{\TateVFunc{}}\\
	\KatRep{A_v}{\AbsGalois{K}}{fp} \ar[d]_-{\omega_{rep}} & \KatRep{Q_v}{\AbsGalois{K}}{cont}  \ar[d]_-{\omega_{rep}} \\
	\KatMod{A_v} \ar[r]_-{\otimes_{A_v}Q_v} & \KatVec{Q_v}}.
\end{xy}
\end{align*} and hence
\begin{align}\label{eq:Lattice=Motivic}
\MG{\M}= {\iMG{\M}}{}\times_{A_v} Q_v
\end{align}  holds for a $ \M \in \KatLshtuka{A_v,R}$; see \cite[Cor. 6.21]{Wedhorn.2004} for further properties.\\
Now consider a finite homomorphism $f:\F{v}\Laurrentreihe{z'}\to \F{v}\Laurrentreihe{z} $ of rings. Let $\gamma': \Fq \Laurrentreihe{z'} \to R$ be the morphism  sending  $z'\mapsto \zeta'$  and let $\gamma:\Fq \Laurrentreihe{z} \to R$ the  one sending $z \mapsto \zeta$. Thus $\gamma'=\gamma\circ f$ holds and we have an induced functor
\begin{align*}
\KatLshtuka{\F{v}\Laurrentreihe{z},R} &\to \KatLshtuka{\F{v}\Laurrentreihe{z'}, R}\\
\M  &\mapsto f_*(\M)
\end{align*}
where $ f_*(\M) =\M $ viewed as a $ R\Laurrentreihe{z'} $-module. The $z'$-rank of the local shtuka $f_*(\M)$ is given by the formula  $ \rk_{z'}(f_*(\M))=\rk_{{z}}\M \cdot \deg f$.\\
It is natural to ask for a functor $f^*:\KatLshtuka{\F{v}\Laurrentreihe{z'}, R}\to \KatLshtuka{\F{v}\Laurrentreihe{z},R} $ induced by the module-tensor product. However, it produces in general not a local shtuka in $\KatLshtuka{\F{v}\Laurrentreihe{z},R}$. For example let $\F{v}=\Fq$ and $f:\Fq\Laurrentreihe{z'}\to \Fq\Laurrentreihe{z}$ given by $z' \mapsto z^2$.
Then $\gamma$ is given by $ z \mapsto \zeta$ and $\gamma'=\gamma\circ f$ by $ z' \mapsto \zeta^2=:\zeta'$. Consider the Carlitz-shtuka $\Carlitz=(R\Laurrentreihe{z'}, z'-\zeta')$ in $\KatLshtuka{\F{q}\Laurrentreihe{z'}, R}$.
Then $f^*(\Carlitz)$ is given by the module $R\Laurrentreihe{z'}\otimes_{\Fq\Laurrentreihe{z'}} \Fq\Laurrentreihe{z}$ and the morphism $z^2-\zeta^2=(z-\zeta)(z+\zeta)$. It is only an isomorphism after inverting $z-\zeta $ and $z+\zeta$ and so not a local shtuka. To solve this issue, it suffice to consider only $z'$ which are positive $\mathrm{char}(\Fq)$-powers of $z$.
In general let $z'= z^{p^n} $ for $p=char(\Fq)$ and $n\in \mathbb{N}$ and define the functor $f^*$ by
\begin{align*}
f^*:\KatLshtuka{\F{v}\Laurrentreihe{z'}, R} &\to \KatLshtuka{\F{v}\Laurrentreihe{z}, R}\\
\M'&\mapsto f^*(\M'):= \M' \otimes_{\F{v}\Laurrentreihe{z'}} \F{v}\Laurrentreihe{z}
\end{align*}
\begin{prop}\label{prop:Adjoint}
	Let $z'=z^{p^n}$ with $n \in \mathbb{N}$ and $p=\mathrm{char}(\Fq)$. Then the functors  $ f^* $ and $ f_* $ dare adjoint, ie. $\Hom{f^*\M}{\N}{\KatLshtuka{\F{v}\Laurrentreihe{z},R}}= \Hom{\M}{f_*\N}{\KatLshtuka{\F{v}\Laurrentreihe{z'},R}}$ holds.
\end{prop}
\begin{proof}
	The functor $ f^* $ interchanges with $ \sigma^* $ as $\sigma$ is the identity restricted to the uniformizing parameter on $A_v$.\\
	Let $ \M' \in \KatLshtuka{\F{v}\Laurrentreihe{z'},R} $ and by extension to $R\Laurrentreihe{z'}$  we get  a homomorphism $ f^*_R: R\Laurrentreihe{z'} \to R\Laurrentreihe{{z}} $ commuting with $f^*$ and $\gamma$. 	The same argument holds for $ f_* $.
	The unit and counit are induced by the $\Rz$ and respectively $R\Laurrentreihe{z'}$-module structures. So it suffices to show their compatibility with the shtuka structure morphism.\\
	Let $ \eta: id_{\KatLshtuka{\F{v}\Laurrentreihe{z'}}} \to f_*\circ f^* $ given by
	\begin{align*}
	\eta_{\M'}:\M'= (M',\tau_{M'}) &\to  f_*\circ f^*(\M'):= (\M'\otimes_{\F{v}\Laurrentreihe{z'}} \F{v}\Laurrentreihe{z},{\tau_{f_*f^*M'}}:=(\tau_{M'}\otimes id))\\
	m &\mapsto m\otimes 1
	\end{align*}
	for an element  $\M' \in \KatLshtuka{\F{v}\Laurrentreihe{z'}} $. Let $m \in \sigma^*\M$. Then $\eta_{M'}(\tau_{M'}(m))=\tau_{M'}(m)\otimes 1$
	and $ \tau_{f_*f^*M'}\circ \sigma^* \eta_{\M'}(m)= (\tau_{M'}\otimes id)(\sigma^* \eta_{\M'}(m))= \tau_{M'}(m)\otimes 1$ holds, as $f^*\sigma^*=(f\otimes \mathrm{id}_{R})^*\circ (\mathrm{id}\otimes \mathrm{Frob}_q)=(f\otimes \mathrm{Frob}_q)^*=\sigma^*f^*$ holds.
	Thus is unit-morphism $\eta_{\M'}$ a shtuka morphism. 	The counit is given by the natural transformation $ \varepsilon:f^*f_* \to id_{\KatLshtuka{\F{v}\Laurrentreihe{z}}} $, which is for a $ \M \in \KatLshtuka{\F{v}\Laurrentreihe{z}} $ given by $ a \otimes m  \mapsto a\circ m $. Then the following diagram
	\begin{align*}
	\begin{xy}
	\xymatrixcolsep{3pc}\xymatrix{
		f^*f_*M \ar[rr]^-{\varepsilon_\M}& &  M  & b\otimes \tau_M(m) \ar@{|->}[r]& b\cdot \tau_M(m)\\
		\underset{= \sigma^*f^*f_*M}{\F{v}\Laurrentreihe{z}\otimes_{\F{v}\Laurrentreihe{z'}} \sigma^*M} \ar[u]^-{id_{\F{v}\Laurrentreihe{z}}\otimes \tau_M}\ar[rr]_-{\varepsilon_{\sigma^*M}}&  &\sigma^*M \ar[u]_-{\tau_M} & b\otimes m \ar@{|->}[u] \ar@{|->}[r]& b\cdot m \ar@{|->}[u]
	}
	\end{xy}
	\end{align*}  commutes.
\end{proof}
Let $  \KatMod{\KatLshtuka{\F{v}\Laurrentreihe{z'},R}}(\F{v}\Laurrentreihe{z}) $ be the category of $ \F{v}\Laurrentreihe{z}$-modules in $\KatLshtuka{\F{v}\Laurrentreihe{z'},R}$. It's objects are pairs $(\M,\phi)$, where $\M \in  \KatLshtuka{\F{v} \Laurrentreihe{z'},R}$
and $ \phi: \F{v}\Laurrentreihe{z}\to \End{R}{\M}$. A morphism of $\F{v}\Laurrentreihe{z}$-modules is a morphism of shtukas in $\KatLshtuka{\F{v}\Laurrentreihe{z'}}$ compatible with $\phi$.
\begin{prop}\label{Prop: Equivalence of sht with modulkat}
	There exists an equivalence of categories between \begin{align*}
	G:\KatMod{\KatLshtuka{\F{v}\Laurrentreihe{z'},R}}(\F{v}\Laurrentreihe{z}) \rightleftarrows \KatLshtuka{\F{v}\Laurrentreihe{z},R}:H
	\end{align*}
\end{prop}
\begin{proof}
	We will construct the functors $G$ and $H$.\\
	We define $G$ as the functor, which sends an element $ (\M'=(M',\tau_{M'}),\phi) $ to  $\M'$, but with $ M' $ considered as a module over the tensor product $ R\Laurrentreihe{z'} \otimes_{\F{v}\Laurrentreihe{z'}} \F{v}\Laurrentreihe{z} = \Rz $ via $\phi$. We have that $\rk_{R\Laurrentreihe{z'}}\M' = [\F{v}\Laurrentreihe{z}:\F{v}\Laurrentreihe{z'}] \cdot \rk_{\Rz}\M'$ holds, which is possible as the extension by $f^*$ is finite.\\
	We define $H$ as the functor, which considers a $ \M \in \KatLshtuka{\F{v}\Laurrentreihe{z}} $ as a $ R\Laurrentreihe{z'} $-module via the ring homomorphism $f$. The morphism   $ \phi:  \F{v}\Laurrentreihe{z} \to \End{}{\M}$ comes from the internal module structure of $ \M $. Here we have
	$ \rk_{R\Laurrentreihe{z'}} \M = \rk_{\Rz} \M \cdot[\F{v}\Laurrentreihe{z}:\F{v}\Laurrentreihe{z'}] $.
\end{proof}
By the remark after definition~\ref{Def: Wedhorn Extension}, $f^*$ and $f_*$ can be extended to  $Q_v$-linear functors on $\KatLshtuka{Q_v,R}$, maintaining their properties. Especially, proposition~\ref{Prop: Equivalence of sht with modulkat} can be extended to $\KatLshtuka{Q_v,R}$. We denote them in the same way.
\begin{prop}\label{prop:Faser von GM}
	Let $z'=z^{p^n}$ for  $n \in  \mathbb{N}$ and $p=\mathrm{char}(\Fq)$. Let $ \M' \in \KatLshtuka{\F{v}\DoppelRund{z'},R} $ and $ f^*: \F{v}\DoppelRund{z'}\to \F{v}\DoppelRund{z} $. Then there exists  an equivalence of algebraic groups  $ \MG{f^*\M'} = \MG{\M'} \times_{\F{v}\DoppelRund{z'}} \F{v}\DoppelRund{z}$ over $\F{v}\DoppelRund{z}$.
\end{prop}
\begin{proof}
	Let $ \M' \in \KatLshtuka{\F{v}\DoppelRund{z'}} $. Then by Tannakian duality $ \DoppelSpitz{\M'}\cong \KatRep{Q_v}{\MG{\M'}}{} $ is an equivalence of categories.
	Denote by $ \DoppelSpitz{f^*\M'} $ the strictly full Tannakian subcategory of $ \KatLshtuka{\F{v}\DoppelRund{z}} $ generated of the essential image of $ f^* $ restricted to $\DoppelSpitz{\M'}$. It follows by the work of Deligne and Milne, cf. \cite[Prop. 3.11]{Deligne.Milne.2018}, and proposition~\ref{Prop: Equivalence of sht with modulkat} that
	\begin{align}\label{eq: equivalenceDM}
	\KatRep{\F{v}\DoppelRund{z}}{\MG{f^*\M'}}{} = \KatRep{\F{v}\DoppelRund{z}}{\MG{\M'}}{}
	\end{align}
	holds. By \cite[Rem. 3.12]{Deligne.Milne.2018} we consider the following situation
	\begin{align*}
	\begin{xy}
	\xymatrix{
		&\text{ \texttt{Vec}}_{\F{v}\DoppelRund{z}}\\
		\KatRep{\F{v}\DoppelRund{z'}}{\MG{\M'}}{}\ar[r]_-{\underline{\TateVFunc{}}} \ar[ru]^-{\omega'}  & \text{\texttt{Vec}}_{\F{v}\DoppelRund{z'}} \ar[u]_-{\otimes_{\F{v}\DoppelRund{z'}} \F{v}\DoppelRund{z} }
	}
	\end{xy}
	\end{align*}
	with $ \omega'(\M)=\underline{\TateVFunc{}}\M\otimes_{\F{v}\DoppelRund{z'}} \F{v}\DoppelRund{z}$ for $ \M \in \DoppelSpitz{\M'} $.  Set  $(\MG{\M'})_{\F{v}\DoppelRund{z}}:=\TensorAut{\omega'}{\DoppelSpitz{\M'}}$.
	This induces again by \cite[Rem. 3.12]{Deligne.Milne.2018}  the equivalence $\KatRep{\F{v}\DoppelRund{z}}{\MG{\M'}}{}\cong \KatRep{\F{v}\DoppelRund{z}}{(\MG{\M'})_{\F{v}\DoppelRund{z}}}{}$.
	Here $  {\MG{\M'}}_{\F{v}\DoppelRund{z}}  $ is the usual base change of algebraic groups.
\end{proof}
Note that in the case of local shtukas the tensor product is already defined over a finite extension. So there is no need to go to  an Ind-completed category, as all objects have a finite basis, in contrary to the general case  studied in the paper of Deligne and Milne \cite{Deligne.Milne.2018}. By that, we mean that the field extension construction of them demand the existence of  an external tensor product $\underline{\otimes}$ with the category  of vector spaces.
By proposition~\ref{Prop: Equivalence of sht with modulkat} the functor $ f_*f^* $ defines $ \KatLshtuka{\F{v}\Laurrentreihe{z}} $ as  a $ \KatMod{\Rz} $-module category,
see \cite[7.1.1]{Etingof.Gelaki.Nikshych.Ostrik.2015} for definition of this category. This structure fulfills the requirements of the external tensor product described above.\\
The association of an algebraic group to a Tannakian lattice is done in a similar sense as Deligne's for fields. So with equation~\eqref{eq:Lattice=Motivic} we can use the proof of proposition~\ref{prop:Faser von GM} also for our Tannakian lattice $ \KatLshtuka{\F{v}\Laurrentreihe{z'}} $  and get the following.
\begin{Corollary}\label{Cor:Faser von LGM}
	Let $ \M' \in \KatLshtuka{\F{v}\Laurrentreihe{z'}} $  and let $z'=z^{p^n}$ for $n\in\mathbb{N}$ and $p=\mathrm{char}(\Fq)$.  Then there exists  an equivalence of algebraic groups  ${\iMG{f^*\M'}}{}={\iMG{\M'}}{} \times_{\F{v}\Laurrentreihe{z'}} \F{v}\Laurrentreihe{z}$.
	Moreover, it is compatible with the field extension to $ \MG{\M'} $.
\end{Corollary}

\section{A condition for openness}
In section~\ref{sec: Tate modules and Galois representation} we constructed a Galois representation
\begin{align*}
\GR{\M}: \AbsGalois{K} \to \HalgGp{\M}(Q_v)
\end{align*}
associated to a local shtuka $\M$ of rank $r$ over $\Spec R$. For a general $\M$, there exists by Tannakian  theory an inclusion $ \HalgGp{\M}\subseteq \MG{\M}\subseteq \mathrm{GL}_r$ of algebraic groups over $Q_v$. We have seen that for all $g \in \AbsGalois{K}$ its image under $\GR{\M}$ is contained in $\mathrm{Aut}_{Q_v}(\TateVFunc{\M})$.
In addition, $\GR{\M}(g)\in \TensorAut{\omega_{Sht_Q}}{\DoppelSpitz{\M}}(Q_v)=\MG{\M}(Q_v)$ holds. So the image of $\AbsGalois{K}$ under $\GR{\M}$ and respectively its Zariski closure $\HalgGp{\M}$ is already contained in $\MG{\M}(Q_v)$ and in $\MG{\M}$, respectively.
If the rank of the local shtuka $\M$  is $1$, this inclusion becomes an equality. In the last section we have computed the integral motivic Galois group $ {\iMG{\M}}$, which is an algebraic group over $A_v$ with   $\MG{\M}= \iMG{\M}\times Q_v$.
Let $K^{un}$ and $R^{un}$ the maximal unramified extension of $K$ and $R$.
\begin{prop}\label{prop:Form local rank 1 shtuka}
	Every local shtuka $\M$ of rank $1$ and virtual dimension $d\in \mathbb{Z}$ over $R^{un}$ is isomorphic to the $d$-th tensor power $\Carlitz^{\otimes d}$ of the  Carlitz-shtuka over $R^{un}$.
\end{prop}
\begin{proof}
	By corollary~\ref{Cor:Form local shtuka} $\M=(R^{un}\Laurrentreihe{z},\tau_M=c \cdot(z-\zeta)^d$. Write $c=\sum_{i=0}^\infty c_i z^i$ with $c_i \in R^{un}$ and $c_0\in (R^{un})^\times$.
	The isomorphism is given by sending $1\in \M$ to $u=\sum_{i=0}^\infty u_i z^i\in \Carlitz^{\otimes d}$ which satisfies $\sigma(u)= c \cdot u$ that is $\sum_{i=0}^\infty u_i^q z^i= \sum_{i=0}^\infty\sum_{j=0}^{i} c_ju_{i-j} z^i$,
	or equivalently $u_0^q=c_0u_0$ and $(u_0^{-1} u_i)^q-u_0^{-1}u_i=\sum_{j=1}^{i}c_0^{-1}c_{i-j}u_0^{-1}u_j$. Since $c_0 \in (R^{un})^\times$ all these equations indeed have solutions  $u_i \in R^{un}$ and $u_0 \in (R^{un})^\times$.
\end{proof}
Note that isogenous shtukas has the same rank and dimension. Furthermore the category $\KatLshtuka{Q_v,R}$ is the category of shtukas except isogeny, so  all isogenous shtukas are isomorphic.
\begin{Theorem}\label{prop: Openness rank one shtuka}
	Let $ \M $ be a local shtuka of rank $1$ and virtual dimension $d>0$ over $R$ and $K=\mathrm{Frac}(R)$. Let $ d'  \in \Z$ be the largest number such that $ d=p^e\cdot d' $ holds for $p=\mathrm{char}(\Fq)$, $ p \nshortmid d' $ and $e\geq 0$.
	Then $\GR{\M}(\AbsGalois{K})$ is contained and open in $\mathrm{GL}_1(\F{v}\Laurrentreihe{z^{p^e}})$. In particular, $\GR{\M}(\AbsGalois{K})$ is open in $\HalgGp{\M}(Q_v)$ if and only if $p\nmid d$.
\end{Theorem}
\begin{proof}
	By proposition~\ref{prop:Form local rank 1 shtuka} $\M\otimes_R R^{un}$ is isomorphic to $\Carlitz^{\otimes d}$ over $R^{un}$. Let $ d'  \in \Z$ be the largest number such that $ d=p^e\cdot d' $ holds with $ p \nshortmid d' $ and $e\geq 0$.\\
	If  $d=d'=1$ holds, then  it follows that   $\GR{\M}(\AbsGalois{K})$ contains the subgroup $\GR{\M}(\AbsGalois{K^{un}})= \GR{\Carlitz}(\AbsGalois{K^{un}})$, since $K^{un}$ is a finite field extension of $ \overline{\F{v}}\DoppelRund{\zeta}$.
	The group $\GR{\Carlitz}(\AbsGalois{K^{un}})$ has finite index in $\GR{\Carlitz}(\AbsGalois{\overline{\F{v}}\DoppelRund{\zeta}})=\F{v}\Laurrentreihe{z}^\times$, see example~\ref{Ex: Carlitz representation}. So the assertion follows.\\
	Let $d\neq 1$ and let $ f:\F{v}\Laurrentreihe{z'}\to \F{v}\Laurrentreihe{z}$ be the homomorphism with $ z'\mapsto z^{p^e}$. Then
	\begin{align} \label{eq: reduction of M}
	\M\otimes_R R^{un}=(R^{un}\Laurrentreihe{z},\tau_{M})	&\cong  (R^{un}\Laurrentreihe{z}, (z-\zeta)^{p^e\cdot d'}) \nonumber\\
	&\cong  (R^{un}\Laurrentreihe{z}, (z^{p^e}-\zeta^{p^e})^{ d'})\nonumber\\
	&\cong f^*(R^{un}\Laurrentreihe{z'}, (z'-\zeta^{p^e})^{ d'})\nonumber\\
	&\cong f^*(\overline{\F{v}}\Laurrentreihe{\zeta'}\Laurrentreihe{z'}, (z'-\zeta')^{ d'}) \otimes_{\overline{\F{v}}\Laurrentreihe{\zeta'}\Laurrentreihe{z'}}R'\Laurrentreihe{z'}
	\end{align}
	holds.
  Set $  \overline{R}:=\overline{\F{v}}\Laurrentreihe{\zeta'} $.
  Let $ K, K^{un} $ and $ \overline{K} $ be the fraction fields of $ R, R^{un} $ and $ \overline{R}$ respectively.
	Then the inclusions	$ \AbsGalois{K'}\subseteq \AbsGalois{\overline{K}} $ is open.
  Let $\underline{C}=(\overline{R}\Laurrentreihe{z'},z'-\eta)$ be the Carlitz shtuka over $\overline{R}$.\\
	By equation~\eqref{eq: reduction of M}, the algebraic group associated to the pullback $f^*\underline{\hat{C}}$ is  a ring extension of the algebraic group $ {\iMG{\underline{\hat{C}}}}$ associated to $\underline{\hat{C}}$ over $\F{v}\Laurrentreihe{z'}$  to $\F{v}\Laurrentreihe{z}$, see corollary~\ref{Cor:Faser von LGM}.
	So the image of the Galois representation is already in  $\iMG{\Carlitz}(\F{v}\Laurrentreihe{z'})=\F{v}\Laurrentreihe{z'}^\times $ and  the following inclusion exits:
	\begin{align*}
	\begin{xy} \xymatrix{
		\AbsGalois{K}\ar[r]^-{\GR{f^*\Carlitz}} \ar[d]_{\GR{\underline{\hat{C}}}} & \iMG{f^*\Carlitz}(\F{v}\Laurrentreihe{z})\\
		\F{v}\Laurrentreihe{z'}^\times=\iMG{\Carlitz}(\F{v}\Laurrentreihe{z'})\ar@{^{(}->}[r] &\iMG{\Carlitz} ( \F{v}\Laurrentreihe{z})=\F{v}\Laurrentreihe{z}^\times, \ar[u]
	}
	\end{xy}
	\end{align*}
	where the right side inclusion is an equality as both are the $\F{v}\Laurrentreihe{z}$-valued points.	So if $p \nmid  d$ and  hence $e\geq 1$ and $\F{v}\Laurrentreihe{z'}^\times=\F{v}\Laurrentreihe{z^{p^e}}^\times \subseteq \F{v}\Laurrentreihe{z}^\times $, the  image of Galois factors through $\F{v}\Laurrentreihe{z'}^\times $.
  We compute the image of $\GR{\Carlitz}\in \F{v}\Laurrentreihe{z'}^\times$.\\
	We denote by $\overline{K}^{un}$ the maximal unramified extension of $\overline{K}$. Let  $\overline{K}_n:= \overline{K}(l_0',\dots,l_n')$ and   $\overline{K}_\infty:= K(l_i':i\in \mathrm{N}_0)$  be the Carlitz-Tate tower of $\overline{K}$,
	ie. the fields extension of $ \overline{K}$ generated by the solution of the Carlitz-Tate polynomials, cf. example~\ref{ex: Carlitz tate module} and example~\ref{Ex: Carlitz representation}.
	Like in \cite[Example 4.10 ]{Hartl.Kim.2015} the extension $\overline{K}_0 $ over $\overline{K}$ is totally tamely ramified of degree $q-1$ and $\overline{K}_n$ over $\overline{K}_{n-1}$ is totally wildly ramified of degree $q$ for $n\geq 1$.
	Like in example~\ref{Ex: Carlitz representation} the field extension to $\overline{K}^{un}$  and its Carlitz-Tate tower $\overline{K}_n^{un}:= \overline{K}(l_0',\dots,l_n')$ and   $\overline{K}_\infty^{un}:= K(l_i':i\in \mathrm{N}_0)$ has the same ramification. So we can only consider the extension of the local shtukas to the maximal unramified extensions.
	By equation~\eqref{eq: reduction of M}, $\M$ is isomorphic to $f^*\Carlitz^{\otimes d}$ over $\overline{R}^{un}$.
	Then $\TateTFunc{\M}\cong (\TateTFunc{f^*\underline{C})}^{\otimes d'} $, as $\TateTFunc{}$ is a tensor functor. So by the definition of the tensor product in $\KatMod{A_v}$, the Tate-module $\TateTFunc{f^*\underline{C}}^{\otimes d'}$ is generated as an
	$A_v$-module by $a= \sum_{n=0}^{\infty} a_n {z'}^n$, which is a $d'$-th power of the inverse basis element  of $\TateTFunc{f^*\underline{C}}$. \\
	Set $a = (l_+')^{d'}$  for $l_+'=\sum_{i=0}^{\infty}l_i' {z'}^i$ and  consider $d'$ indices $i_1,\dots,i_{d'}$. Define $j_m:= i_m-i_{m+1}$ for $1\leq m\leq d'$ and set $n=i_1= j_1+\dots +j_d$.
	Set $I_n:=\{(j_1,\dots ,j_d)\in \mathbb{N}_0^{d'} |\ j_1+\dots +j_d=n\}$ and $I_n^{<}:=\{(j_1,\dots ,j_d)\in I_n |\ j_m<n \ \forall m=1,\dots, d'\}$.
	Then we have {by iteration} of the  product formula for formal power series
	\begin{align}\label{eq: potenz von laurrent}
	a&= \sum\limits_{n=0}^{\infty} a_n {z'}^n = (\sum\limits_{i=0}^{\infty} l_i {z'}^i)^{d'}\nonumber\\
	&= \sum\limits_{i_1=0}^{\infty}(\sum\limits_{i_2=0}^{i_1}(\dots (\sum\limits_{i_{d'}=0}^{i_{d'-1}} l_{i_{d'}}'\cdot l_{i_{d'-1}-i_{d'}}' \cdot \dots \cdot l_{i_{1}-i_{2}}')) {z'}^{i_1} \nonumber\\
	&= \sum\limits_{n=0}^{\infty} (d' {l_0'}^{d'-1}\cdot l_n'+ \sum\limits_{(j_1,\dots ,j_{d'})\in I_n^<}l_{j_1}'\cdot \dots \cdot l_{j_{d'}}') {z'}^n
	\end{align}
	Set $ \overline{K}^{un,d'}_n:= \overline{K}^{un}(a_0\dots a_n)$. The coefficient $l_n'$ is prime to $ \sum_{(j_1,\dots ,j_{d'})\in I_n^<}l_{j_1}'\cdot \dots \cdot l_{j_{d'}}')$ for all $n$ and hence $\overline{K}_n^{un}= \overline{K}^{un, d'}_n(l_0)$ holds.
	Then $l_0^{d'}-a_0$ is divided by  the minimal polynomial of the field extension $\overline{K}^{un}_0$ over $ \overline{K}^{un,d'}_0$, which implies
	$[ \overline{K}^{un}_n: \overline{K}^{un,d'}_n]\leq d'$. Then  $ \overline{K}^{un}_n$ is Galois over $\overline{K}^{un,d'}_n$ and  we consider the following diagram induced by Galois theory:
	\begin{align*}
	\begin{xy}
	\xymatrix{
		\Gal(\overline{K}^{un}_n/\overline{K}^{un}) \ar@{->>}[r]^-f   \ar[d]_-{\rho_{\underline{C}}}^-{\cong}&
		\Gal (\overline{K}^{un,d'}_n/\overline{K}^{un}) \ar[d]^-{\rho_{\underline{C}^{\otimes d'}}} &
		g \ar@{|->}[r]\ar@{|->}[d]&
		**[r] g\restrict{\overline{K}^{un,d'}_n} \ar@{|->}[d]\\
		(\F{v}\Laurrentreihe{z'}/({z'}^{n+1}))^\times \ar@{->>}[r]^-{(x\mapsto x^{d'})}_-{g}&
		((\F{v}\Laurrentreihe{z'}/({z'}^{n+1}))^\times)^{d'} &
		\GR{\Carlitz}(g)= \frac{g(l_+)}{l_+} \ar@{|->}[r]&
		\frac{g(l_+)}{l_+}^{d'}= \frac{g(a)}{a}.
	}
	\end{xy}
	\end{align*}
	The morphism $\rho_{\underline{C}^{\otimes d'}}$ is surjective as $(x\mapsto x^{d'})\circ\GR{\Carlitz}$ is surjective. Let  $g(a)=a$, then $g(a_i)=a_i$ holds for all $i$. But then $g= \mathrm{id}$. So $\rho_{\underline{C}^{\otimes d'}}$ is also an isomorphism. It remains to show that $ \rho_{\underline{C}^{\otimes d'}}(\mathrm{Gal} (\overline{K}^{un}_\infty/ \overline{K}^{un} )) \subseteq \F{v}\Laurrentreihe{z'}^\times$ is open with index $\leq d'$.\\
	We know  by Galois theory that $\Ker \ f = \mathrm{Gal}( \overline{K}^{un}_n/ \overline{K}^{un,d'}_n)$. So the number of elements in the kernel is lesser equal $d'$. Then  $\Ker \ g$ has also lesser equal $d'$ number of elements. $\Ker \ g$ is a subset of $(\F{v}\Laurrentreihe{z'}/({z'}^{n+1}))^\times$ which consists $(q_v-1)\cdot q_v^n$ elements.
	Then $\#\mathop{\mathrm{im}} g\geq \frac{(q_v-1)\cdot q_v^n }{d'}$. Therefore the index of the image is lesser equal $d'$ for all $n$.
\end{proof}
\begin{rem}
	Let $\M$ be a local shtuka over $R$ with fraction field $K$. We call the field towers given by a basis of the Tate-module \emph{Tate tower of $\M$}  and the tower given by the Carlitz Tate polynomials \emph{Carlitz-Tate towers}. Classically the last one is denoted by \emph{$v$-adic cyclotomic tower of $K$}. In the same way,  we call the field extensions \emph{Tate field extensions of $K$ generated by $\M$} and \emph{Carlitz-Tate field extensions of $K$}.\\
	The above computation shows that if we are in the open image situation, the Carlitz-Tate field extensions of $K$  gives a upper bound for all Tate field extensions of $K$ generated by $\M$ of rank $1$.
\end{rem}
\begin{ex}[squared Carlitz local shtuka]
	Consider the Carlitz shtuka defined by $\underline{\hat{C}}= (\Rz,(z-\zeta))$ over $ R:= \F{q}\Laurrentreihe{\zeta}$ and set $ K=\F{q}\DoppelRund{\zeta} $. Let $l_+= \sum_{i=0}^{\infty}l_i z^i  \in \MO{K^{sep}}\Laurrentreihe{z}$ be the solutions of the Carlitz-Tate polynomials, cf. equation~\eqref{eq: Carlitz-Tate equation}.
	The squared Carlitz shtuka is defined as $\underline{\hat{C}}^{\otimes2}:=(\Rz,(z-\zeta)^2) $ and by proposition~\ref{prop: Openness rank one shtuka} it's Tate-module is given by $\TateTFunc{\Carlitz^{\otimes2}}= a^{-1} \cdot\F{q}\Laurrentreihe{z} $ for $ a=l_+^{2}$.
	After solving this equation, the element $a = \sum_{i=0}^{\infty}a_i z^i$ given by $ a_0=l_0^2 $ and $ a_n=2l_0l_n+ \sum_{i=1}^{n-1}l_il_ {n-i}$ for $n>0$.
	Set $K_n:=K(l_0,\dots, l_n)$ and $ K^2_n:=K(a_0,\dots a_n)$. Then  $K^2_n(l_0) =K_n$ holds, as $ l_n$ fulfills $l_n= \frac{1}{2l_0}(a_n-\sum_{i=1}^{n-1}l_il_ {n-i}) $.
	It follows that $ \mathrm{minpo}_{l_0/ K^2_0}  $ has to divide $ l_0^2-a_0 $. Furthermore, it equals $ l_0^2-a_0 $, as the extension of $ K^2_0$ by $l_0$ is a real extension. To show this, let $ q=3 $. Then $a_0^2=\zeta^2$ and we have that $ a_0={l_0'}^{q-1}=\zeta $ holds. Hence
	$ K^2_0 (l_0)=K$ is a subfield of $ K_0 $ with $ [K^2_0:K]=(q-1)=2 $. For $ q>3 $, $ a_0 $ is a $ (q-1)/2=:d $-root of $ -\zeta $. The field extension generated by roots cannot contain $ 2 $-roots and hence $ l_0 \notin K^2_0 $.
	So  $[K_n:K^2_n]= 2$ and by Galois theory $[K^2_n:K^2_{n-1}]=q  $ holds for all $ n > 0 $.\\
	Now consider the Galois representation $ \GR{\Carlitz^{\otimes2}}= (\GR{\Carlitz})^{\otimes 2} $ and the induced diagram
	\begin{align*}
	\begin{xy} \xymatrix{
		\Gal(K_n/\Fqrzeta) \ar[d]_-{\dsim} \ar@{->>}[r] & \Gal(K^2_n/\Fqrzeta)\ar[d]^-{\dsim} \\
		\Fqz^\times/(z^{n+1}) \ar@{->>}[r] &( \Fqz^\times/(z^{n+1}))^2
	}
	\end{xy}
	\end{align*}
	for all $n \geq 0$. Here, the ring $ ( \Fqz^\times/(z^{n+1}))^2 $ is given by the elements $ \sum_i^n b_i z^i  $ in $ \Fq[z]^\times $  with $ b_0 \in (\F{q}^\times)^2 $.
	Then $ \Fqz^\times/ (\Fqz^\times)^2=\F{q}^\times/(\F{q}^\times)^2=\{1,\alpha\} $ with $ \alpha \in  \F{q}^\times \setminus (\F{q}^\times)^2 $ holds and it follows that the squared group has index $ 2 $ in the original. So it is an open subgroup, which implies the openness of the Galois representation.\\
	Also by proposition~\ref{prop: Openness rank one shtuka}, we have the Galois openness, as $d'=d=2$ holds for $\underline{\hat{C}}^{\otimes2}$ and we assumed in the beginning that $p$ is odd.
\end{ex}
\begin{ex}
	We assume the same setup as in the previous example and  consider the local shtuka $ \M:=(\Rz,\tau_M=(z^q-\zeta^q)) $ of rank $1$ of dimension $q$ over $R$. Then the motivic Galois group is given by $\GM{}$. Thus by the previous proposition the image is given by $\F{q}\Laurrentreihe{z^q}^\times $, which is not $v$-adically open in $\F{q}\Laurrentreihe{z}^\times $.
\end{ex}
\begin{prop}\label{prop: openes for higher r}
	Let $\M$ be a local shtuka of rank $r\geq1$  and dimension $d$ over $R$.\\
	If $\mathrm{char}(\F{q}) | d$ then ${\GR{\M}(\AbsGalois{K})}\subseteq \mathrm{GL}_r(\F{v}\Laurrentreihe{z})$ is not open.
\end{prop}
\begin{proof}
	Let $\M$ be a local shtuka of rank $r\geq1$ and dimension $d$ over $R$. We denote by $\Lambda^r\M:= (\Rz, \mathrm{det}\tau_M)$ the determinant of $\M$. $\Lambda^r\M$ has rank $1$ but the same dimension $d$ as $\M$. Set $p=\mathrm{char}(\F{q})$ and let $ d'  \in \Z$ be the largest number such that $ d=p^e\cdot d' $ holds for $ p \nmid d' $ and $e\geq 0$.
	Let $\GR{\M}:\AbsGalois{K}\to \mathrm{GL}_r(\F{v}\Laurrentreihe{z})$  be the $v$-adic Galois representation attached to $\M$ and $\GR{\Lambda^r\M}:\AbsGalois{K}\to \F{v}\Laurrentreihe{z}^\times$ the $v$-adic Galois representation attached to $\Lambda^r\M$.
	As $\Lambda^r\M$ has rank $1$, $\GR{\Lambda^r\M}$ factors by proposition~\ref{prop: Openness rank one shtuka} through $\F{v}\Laurrentreihe{z^{p^e}}^\times$ and we can consider the following {commutative} diagram
	\begin{align}\label{eq: Digaram determinante offen}
	\begin{xy}
	\xymatrix{
		\AbsGalois{K} \ar[r]^-{\GR{\M}} \ar[d]_-{\GR{\Lambda^r\M}} & \mathrm{GL}_r(\F{v}\Laurrentreihe{z}) \ar[d]^-{\mathrm{det}}\\
		\F{v}\Laurrentreihe{z^{p^e}}^\times \ar[r]_-{(*)}   & \F{v}\Laurrentreihe{z}^\times
	}
	\end{xy}
	\end{align}
	Assume $\GR{\M}$ is open.  Then there exists a standard open subset $ 1+z^m\F{v}\Laurrentreihe{z}^{r\times r}$ in $ \mathrm{GL}_r(\F{v}\Laurrentreihe{z})$  for $m>> 0$, onto which a standard open set of $\AbsGalois{K}$ is sent. Let $h \in \F{v}\Laurrentreihe{z}$ and
	\begin{align*}
	\underline{h}=
	\begin{pmatrix}
	1+z^m \cdot h & 0       & \dots  & 0 \\
	0             & 1       & \ddots & \vdots   \\
	\vdots        & \ddots  & \ddots & 0  \\
	0             & \dots   & 0      & 1 \\
	\end{pmatrix}
	\in 1+z^m\F{v}\Laurrentreihe{z}^{r\times r}
	\end{align*}
	Then  $\mathrm{det}(\underline{h})= 1+z^m \cdot h \in \F{v}\Laurrentreihe{z}^\times$ holds, which is again inside a standard open set. So it follows that the morphism $(*)$ of diagram~\eqref{eq: Digaram determinante offen} has to be an open inclusion and furthermore $\GR{\Lambda^r\M} \subseteq \F{v}\Laurrentreihe{z}^\times$ is open. This implies with proposition~\ref{prop: Openness rank one shtuka} $d=d'$ and $p^e =1$.
\end{proof}
The result of theorem~\ref{prop: Openness rank one shtuka} ca be transferred to  \Am s in the following way:
\begin{Theorem}
	Let $ \underline{M} $ be a   \Am \ of  rank $1$ and dimension $d$ over a finite field extension $L$ over $Q$. Let $ d'  \in \Z$ be the largest number such that $ d=p^e\cdot d' $ holds for $p=\mathrm{char}(\Fq)$, $ p \nshortmid d' $ and $e\geq 0$.
	Then $\GR{\underline{M}}(\AbsGalois(L))$ is contained and open in $\mathrm{GL}_1(\F{v}\Laurrentreihe{z^{p^e}})$.
	In particular, $\GR{\underline{M}}(\AbsGalois{L})$ is open in $\HalgGp{\underline{M}}(Q_v)$ if and only if $p\nmid d$.
\end{Theorem}
\begin{proof}
	Let  $L_w$ be the completion of $L$ at  a place $w$ above $ v $.   Then the absolute Galois group $\AbsGalois{L_w}$ of $L_w$ is a subgroup of the absolute Galois group $\AbsGalois{L}$ of $L$.
	$\AbsGalois{L}$ is generated by the groups $g \AbsGalois{L_w} g^{-1}$ for all  $g\in \AbsGalois{L}$.\\
	Then the image of $\GR{\M_v(\underline{M})}$ is inside $\mathrm{Aut}_{A_v}(H^1_v(\M_v(\underline{M})),A_v)= \mathrm{Aut}_{A_v}(H^1_v(\underline{M}),A_v)= \F{v}\Laurrentreihe{z}^\times$. The latter one is the image of $\GR{\underline{M},v}$.
	For all $g \in \AbsGalois{L}$, $\GR{\underline{M},v}(g \AbsGalois{L_w} g^{-1})= \GR{\underline{M},v}( \AbsGalois{L_w})$ is invariant, as $\F{v}\Laurrentreihe{z}^\times$ is commutative.
	So by \cite[\textsc{II}, Prop. 8.5]{Neukirch.1999} $\GR{\underline{M},v}(\AbsGalois{L})$ is generated by the groups $\GR{\underline{M},v}( \AbsGalois{L_w})$. Then the assertion follows by theorem~\ref{prop: Openness rank one shtuka}.
\end{proof}
Furthermore, if we use the same reducing technique as in the previous proof we can extend the result also to \Am s of higher rank, compare~\ref{prop: openes for higher r}:
\begin{Corollary}\label{Cor: Amot is not open}
	Let $ \underline{M} $ be a  \Am \ of rank $r$  and  dimension $d$ over a finite field extension $L$ over $Q$.
	If $\mathrm{char}(\F{v}) | d$ then ${\GR{\M}(\AbsGalois{K})}\subseteq \mathrm{GL}_r(\F{v}\Laurrentreihe{z})$ is not open.
\end{Corollary}
\bibliographystyle{plain}
\bibliography{SNOPUBLISH}

\end{document}